\documentclass{amsart}
\usepackage{amsthm}
\usepackage{amsmath}
\usepackage{amssymb}
\usepackage{euscript}
\usepackage{graphicx}
\usepackage{pgf,tikz}
\usetikzlibrary{arrows}

\newtheorem{theorem}{Theorem}

\newtheorem{corollary}[theorem]{Corollary}

\newtheorem{lemma}[theorem]{Lemma}
\newtheorem{remark}[theorem]{Remark}

\newtheorem{definition}[theorem]{Definition}

\newcommand{\N}{\mathbb{N}}

\newcommand{\cont}{\mathcal C}

\newcommand{\surj}{\mathcal S}

\newcommand{\seq}[1]{\langle #1\rangle}

\newcommand{\chainacc}[1]{\mathcal{CA}(#1)}

\makeatletter
\@namedef{subjclassname@2020}{%
	\textup{2020} Mathematics Subject Classification}
\makeatother

\begin{document}

\title{Shadowing, recurrence, and rigidity in dynamical systems}

\author[J. Meddaugh]{Jonathan Meddaugh}
\address[J. Meddaugh]{Baylor University, Waco TX, 76798}
\email[J. Meddaugh]{Jonathan\_Meddaugh@baylor.edu}

\subjclass[2020]{37B65, 37B45}
\keywords{shadowing, pseudo-orbit tracing property, topological dynamics}

\begin{abstract} In this paper we examine the interplay between recurrence properties and the shadowing property in dynamical systems on compact metric spaces. In particular, we demonstrate that if the dynamical system $(X,f)$ has shadowing, then it is recurrent if and only if it is minimal. Furthermore, we show that a uniformly rigid system $(X,f)$ has shadowing if and only if $X$ is totally disconnected and use this to demonstrate the existence of a space $X$ for which no surjective system $(X,f)$ has shadowing. We further refine these results to discuss the dynamics that can occur in spaces with compact space of self-maps.
\end{abstract}

\maketitle

\section{Introduction}

A dynamical system is a pair $(X,f)$ where $X$ is a compact metric space and $f:X\to X$ is a continuous function. It is well-known that every dynamical system has at least one \emph{recurrent point}---a point $z$ such that for every neighborhood $U$ of $z$, the orbit of $z$ meets $U$ infinitely often. Recurrent points are well-studied objects, with connections to many other dynamical notions.

Katznelson and Weiss began the study of systems in which \emph{every} point is recurrent (called recurrent systems) and demonstrated that there exist transitive non-minimal recurrent systems \cite{KatznelsonWeiss}. Informed by analogous measure-theoretic notions, Glasner and Maon introduced stronger notions of recurrence---$n$-rigidity, weak rigidity, rigidity, and uniform rigidity (in increasing order of strength)---and showed that rigid systems have zero topological entropy \cite{GlasnerMaon}. Korner shows that rigidity and uniform rigidity are distinct notions, even in the restricted context of minimal systems \cite{Korner}. However, in \cite{DonosoShao:uniformlyrigidmodels}, Donoso and Shao demonstrate that the class of uniformly rigid systems is rich enough to model all non-periodic ergodic rigid systems.

In the following sections, we will explore the implications of the various forms of recurrence in systems with shadowing. Informally, a dynamical system $(X,f)$ has the \emph{shadowing property} provided that approximate orbits (pseudo-orbits) are well-approximated by true orbits of the system. Shadowing (also known as the \emph{pseudo-orbit tracing property}) is naturally important in computation \cite{Corless, Pearson} but has also been the subject of significant research in its own right, beginning with Bowen's analysis of non-wandering sets for Axiom A diffeomorphisms \cite{Bowen}.  The shadowing property has wide-reaching applications in the analysis of dynamical systems, especially in analyzing stability \cite{Pil, robinson-stability, walters} and in characterizing $\omega$-limit sets \cite{Bowen,  MR}. Despite being a relatively strong property, shadowing is a generic property in the space of dynamical systems on many spaces \cite{BMR-Dendrites, Mazur-Oprocha, Meddaugh-Genericity, Pilyugin-Plam}.

The main results of this paper are the following theorems.

{
	\renewcommand*{\thetheorem}{\ref{recurrentminimal}}
	\addtocounter{theorem}{-1}
\begin{corollary}
	Let $X$ be connected and let $(X,f)$ be a system with shadowing. Then $(X,f)$ is minimal if and only if it is recurrent.
\end{corollary}
}

{
	\renewcommand*{\thetheorem}{\ref{rigidsystem}}
	\addtocounter{theorem}{-1}
	\begin{theorem} 
		A uniformly rigid system $(X,f)$ has shadowing if and only if $X$ is totally disconnected.
	\end{theorem}
}

{
	\renewcommand*{\thetheorem}{\ref{HnNondense}}
	\addtocounter{theorem}{-1}
	\begin{theorem} 
		There exist continua with the property that shadowing is not dense in $\cont(X)$. In particular, for each $n\in\N$, the continuum $H_n$ has the property that shadowing is not dense in $\cont(X)$.
	\end{theorem}
}

{
	\renewcommand*{\thetheorem}{\ref{compactC}}
	\addtocounter{theorem}{-1}
\begin{corollary}
	Let $X$ be a compact, connected metric space with $\cont(X)$ compact. Then $(X,f)$ has shadowing if and only if $\seq{f^n}$ converges to a constant map.
\end{corollary}
}





%

The organization of this paper is as follows. In Section \ref{Prelim}, we define notation and relevant terminology. In Section \ref{Rigid systems}, we demonstrate that, for systems with the shadowing property, recurrence properties have strong implications on the dynamics of the system and the topological structure of the underlying space. In Section \ref{Rigid Spaces}, we use these implications to demonstrate that there is a large class of spaces on which no surjective map has shadowing. Finally, in Section \ref{non-surjective}, we expand on these ideas to demonstrate that there are spaces for which only the constant maps have shadowing.

\section{Preliminaries} \label{Prelim}
In the following material, we will most frequently be indexing sets and sequences with the non-negative integers ($\omega$) or subsets thereof. It will, however, at times be convenient to index with the positive integers ($\N$) instead.

For the purposes of this paper, a \emph{dynamical system} is a pair $(X,f)$ consisting of a compact metric space $(X,d)$ and a continuous function $f:X\to X$.

For a dynamical system $(X,f)$ and a point $x\in X$, the \emph{orbit of $x$ under $f$} (or, the $f$-orbit of $x$) is the sequence $\seq{f^i(x)}_{i\in\omega}$ where $f^0$ denotes the identity. For $\delta>0$, a \emph{$\delta$-pseudo-orbit for $f$} is a sequence $\seq{x_i}_{i\in\omega}$ such that for all $i\in\omega$, $d(f(x_i),x_{i+1})<\delta$. For $\epsilon>0$, we say that a sequence $\seq{x_i}_{i\in\omega}$ in $X$ \emph{$\epsilon$-shadows} a sequence $\seq{y_i}_{i\in\omega}$ in $X$ provided that $d(x_i,y_i)<\epsilon$ for all $i\in\omega$. 

\begin{definition}
A dynamical system $(X,f)$ has \emph{shadowing} (or the \emph{pseudo-orbit tracing property}, sometimes denoted \emph{POTP}) provided that for all $\epsilon>0$, there exists $\delta>0$ such that if $\seq{x_i}$ is a $\delta$-pseudo-orbit for $f$, then there exists $z\in X$ such that for all $i$, $d(x_i,f^i(z))<\epsilon$, i.e. the orbit of $z$ {$\epsilon$-shadows} the pseudo-orbit $\seq{ x_i}$.
\end{definition}


It is often useful to consider finite versions of pseudo-orbits. For a dynamical system $(X,f)$, $\delta>0$, and $a,b\in X$, a \emph{$\delta$-chain from $a$ to $b$} is a finite sequence $\seq{x_i}_{i\leq n}$ with $x_0=a$, $x_n=b$ and such that for all $i<n$, $d(f(x_i),x_{i+1})<\delta$. Note that any truncation of a $\delta$-pseudo-orbit is a $\delta$-chain and that any $\delta$-chain can be extended to a $\delta$-pseudo-orbit by defining $x_{n+j}=f^j(x_n)$ for all $j\in\omega$. For $a\in X$, the \emph{chain accessible set of $a$} is the set $\chainacc{a}$ consisting of all $b\in X$ such that for all $\delta>0$, there exists a $\delta$-chain from $a$ to $b$.

In a dynamical system $(X,f)$ a point $x\in X$ is \emph{recurrent} provided that for all $\epsilon>0$, there exists $n>0$ such that $f^n(x)\in B_\epsilon(x)$. The system $(X,f)$ is a \emph{recurrent system} provided that each $x\in X$ is recurrent. 
 Stronger still, $(X,f)$ is \emph{uniformly rigid} provided that for each $\epsilon>0$, there exists $n>0$ such that for each $x\in X$, $f^n(x)\in B_\epsilon(x)$ (note that this property could sensibly called uniformly recurrent, but there is an extant unrelated pointwise notion of uniform recurrence in the literature). Finally, a system $(X,f)$ is \emph{periodic} if there exists $n>0$ such that $f^n=id_X$.

It is worth noting that the notions of recurrence and uniform rigidity can be equivalently written in terms of convergence in $X$ and the space $\cont(X)$ of continuous self-maps on $X$ with topology generated by the supremum metric
\[\rho(f,g)=\sup_{x\in X}d(f(x),g(x)).\]
It is clear that convergence with respect to $\rho$ is the same as uniform convergence---this topology is sometimes called the topology of uniform convergence.

\begin{remark}
	Let $(X,f)$ be a dynamical system.
	\begin{enumerate}
		\item $x\in X$ is recurrent if and only if there is a subsequence of $\seq{f^i(x)}$ which converges to $x$.
		\item $(X,f)$ is uniformly rigid if and only if there is a subsequence of $\seq{f^i}$ which converges uniformly to $id_X$. 
	\end{enumerate}
\end{remark}

\section{Recurrence and Rigidity in Systems with Shadowing} \label{Rigid systems}

In this section, we explore the connections between, recurrence, rigidity, shadowing and the topology of the space $X$. Much of this is inspired by the following observation about the identity map.

\begin{theorem} \label{identity}
	The system $(X,id_X)$ has shadowing if and only if $X$ is totally disconnected.
\end{theorem}

\begin{proof}
	Suppose that $(X,id_X)$ has shadowing but that $X$ has a non-trivial component $C$. Fix $\epsilon>0$ less than the diameter of $C$ and let $a,b\in C$ with $d(a,b)>\epsilon$. Since the system has shadowing, choose $\delta>0$ such that every $\delta$-pseudo-orbit is $\epsilon/2$-shadowed.
	
	Since $C$ is connected, we can find a sequence $\seq{x_i}_{i\in\omega}$ in $C$ and $n>0$ with $x_0=a$, $x_k=b$ for all $k\geq n$ and $d(x_i,x_{i+1})<\delta$ for all $i\in\omega$ (See 4.23 of \cite{Nadler}). As this is a $\delta$-chain, choose $z\in X$ whose orbit shadows $\seq{x_i}$. But then $d(a,b)\leq d(a,z)+d(z,b)=d(a,z)+d(f^n(z),b)<\epsilon$, a contradiction. 
	
	Conversely, suppose that $X$ is totally disconnected and fix $\epsilon>0$. Since $X$ is compact, metric and totally disconnected, we can find a finite cover $\mathcal U=\{U_j:j\leq n\}$ consisting of pairwise disjoint clopen sets of diameter less than $\epsilon$. Fix $0<\delta<\min\{d(U_j,U_k):j\neq k\}$ and let $\seq{x_i}_{i\in\omega}$ be a $\delta$-pseudo-orbit for $id_X$. By choice of $\delta$, if $x_i\in U_j$, then $d(x_i,x_{i+1})=d(f(x_i),x_{i+1})<\delta$. Thus $d(x_{i+1},U_j)<\delta$, and hence $x_{i+1}\in U_j$ as well. Thus there exists $j\leq n$ with $\{x_i:i\in\omega\}\subseteq U_j$. Fix $z\in U_j$ and observe that for each $i\in\omega$, $d(x_i,z)$ is less than the diameter of $U_j$, and thus $d(f^i(z),x+i)=d(z,x_i)<\epsilon$, i.e. the orbit of $z$ $\epsilon$-shadows $\seq{x_i}$. Thus $(X,id_X)$ has shadowing.
\end{proof}

This result is easily extended to systems which are periodic. First let us observe the following general result.

\begin{lemma} \label{iterates}
	If $(X,f)$ has shadowing, then for each $n>0$, $(X, f^n)$ also has shadowing.
\end{lemma}

\begin{proof}
	Suppose that $(X,f)$ has shadowing and fix $\epsilon>0$. Choose $\delta>0$ such that every $\delta$-pseudo-orbit for $f$ is $\epsilon$-shadowed by an $f$-orbit. 
	
	Now, let $\seq{x_i}_{i\in\omega}$ be a $\delta$-pseudo-orbit for $f^n$. Define a sequence $\seq{y_j}_{j\in\omega}$ as follows. For $j\in\omega$, let $q_j, r_j\in\omega$ be the unique choices with $r_j<n$ and $j=q_jn+r_j$, and define $y_j=f^{r_j}(x_{q_j})$. Observe that for $j\in\omega$ such that $r_j\neq n-1$, we have $q_{j+1}=q_j$ and $y_{j+1}=f^{r_{j+1}}(x_{q_j})=f^{r_{j}+1}(x_{q_j})=f(y_j)$ and hence $d(y_{j+1},f(y_j))=0<\delta$. For $j$ such that $r_j=n-1$, we have $q_{j+1}=1+q_j$ and $r_{j+1}=0$, and hence $d(y_{j+1},f(y_{j})=d(x_{1+q_j},f(f^{n-1}(x_{q_j}))=d(x_{1+q_j},f^{n}(x_{q_j})<\delta$ since $\seq{x_i}$ is a $\delta$-pseudo-orbit for $f^n$.
	
	In either case, we see that $d(y_{j+1},f(y_j))<\delta$, so we can choose $z\in X$ such that $d(f^j(z),y_j)<\epsilon$ for all $j\in\omega$. But then $d(f^{in}(z),y_{in})<\epsilon$ for all $i$ and, since $y_{in}=x_{q_{in}}=x_i$, we have $d((f^n)^i(z),x_i)<\epsilon$, i.e. the $f^n$-orbit of $z$ $\epsilon$-shadows $\seq{x_i}$.
\end{proof}

\begin{theorem} \label{periodic}
	A periodic system $(X,f)$ has shadowing if and only if $X$ is totally disconnected.
\end{theorem}

\begin{proof}
	Let $(X,f)$ be a periodic dynamical systems and let $n>0$ with $f^n=id_X$.
	
	By the above lemma, if $(X,f)$ has shadowing, then so does $(X,f^n)$, i.e. $(X, id_X)$ has shadowing, and hence $X$ is totally disconnected by Theorem \ref{identity}.
	
	Conversely, suppose that $X$ is totally disconnected and fix $\epsilon>0$. Since $X$ is compact, metric and totally disconnected, we can find a finite cover $\mathcal U=\{U_j:j\leq N\}$ consisting of pairwise disjoint clopen sets of diameter less than $\epsilon$.

	Fix $0<\eta<\min\{d(U_j,U_k):j\neq k\}$ and by uniform continuity of $f$ and its iterates, choose $0<\delta<\eta$ such that if $d(a,b)<\delta$, then $d(f^i(a),f^i(b))<\eta$ for $i\leq n$. However, since $f^n=id_X$, it follows that $d(f^i(a),f^i(b))<\eta$ for all $i\in\omega$.
	
	
	Now, let $\seq{x_i}_{i\in\omega}$ be a $\delta$-pseudo-orbit for $f$. We claim that the point $x_0$ $\epsilon$-shadows this pseudo-orbit. To observe this, notice that for each $j,k\in\omega$, we have $d(f^k(f(x_j)),f^k(x_{j+1}))<\eta$ since $d(f(x_i),x_{i+1})<\delta$. In particular, for each $i$, the sequence $f^i(x_0),f^{i-1}(x_1),\ldots f(x_{i-1}),x_i$ has the property that adjacent terms are within $\eta$ of one another, and by choice of $\eta$, there exists an index $j(i)$ such that each term belongs to $U_{j(i)}$, and hence $d(f^i(x_0),x_i)<\epsilon$. This establishes that $x_0$ $\epsilon$-shadows the pseudo-orbit as claimed and therefore $(X,f)$ has shadowing.
\end{proof}

Of course, periodicity in the sense that $f^n=id_X$ is a very strong form of recurrence. It is then natural to consider whether weaker forms of recurrence have similar implications. We begin by considering recurrent systems, in which each point in the system is recurrent, but there is no guaranteed synchronization of return times. 

\begin{lemma} \label{reversingchains}
	Let $(X,f)$ be a recurrent system. Then for all $\delta>0$, there exists $\eta>0$ such that if $x,y\in X$ with $d(f(x),y)<\eta$, then there exists a $\delta$-chain from $y$ to $x$.
\end{lemma}

\begin{proof}
	Suppose that $(X,f)$ is recurrent and let $\delta>0$. By uniform continuity of $f$, fix $\eta>0$ such that if $d(a,b)<\eta$, then $d(f(a),f(b))<\delta$.
	
	Now, let $x,y\in X$ such that $d(f(x),y)<\eta$. Since $(X,f)$ is recurrent, we can choose $M>2$ such that $d(f^M(x),x)<\delta$. Now, observe that $d(f(y),f^2(x))<\delta$ by choice of $\eta$, and thus, if we define $a_i$ for $i\leq M-1$ as follows:
	\[a_i=\begin{cases} 
		y & i = 0 \\
		f^{i+1}(x) & 0<i<M-1 \\
		x & i=M-1
	\end{cases},\]
	we see that $\seq{a_i}_{i\leq M}$ is a $\delta$-chain from $y$ to $x$.
\end{proof}

\begin{lemma} \label{equiv}
	Let $(X,f)$ be a recurrent system. Then chain accessibility is an equivalence relation on $X$. Furthermore, for all $p\in X$, the  component of $p$ in $X$ is contained in  $\chainacc{p}$.
\end{lemma}

\begin{proof}
	First, note that chain accessibility is transitive in all dynamical systems. Indeed, if $r\in\chainacc{q}$ and $q\in\chainacc{p}$, then for $\delta>0$ there exist $\delta$-chains $\seq{a_i}_{i\leq n}$ from $p$ to $q$ and $\seq{b_i}_{i\leq m}$ from $q$ to $r$. By concatenating these chains, i.e. defining a chain $\seq{c_i}_{i\leq n+m}$ by $c_i=a_i$ for $i\leq m$ and $c_i=b_{i-n}$ for $i\geq n$, we see that there is a $\delta$-chain from $p$ to $r$.
	
	Now, suppose that $(X,f)$ is a recurrent system and let $p\in X$. It is straightforward to see that $p\in\chainacc{p}$: fix $\delta>0$ and, by recurrence, find $M>0$ such that $d(f^M(p),p)<\delta$. It is easy to see that the sequence $p,f(p),\ldots f^{M-1}(p),p$ is the desired $\delta$-chain.
	
	Now, suppose that $q\in\chainacc{p}$. Fix $\delta>0$ and choose $\eta>0$ as in Lemma \ref{reversingchains}. Since $q\in\chainacc{p}$, let $\seq{a_i}_{i\leq n}$ be a $\eta$-chain with $a_0=p$ and $a_n=q$. We wish to demonstrate that there is a $\delta$-chain from $q$ to $p$. By our choice of $\eta$, there exists, for each $0<i\leq n$, a $\delta$-chain from $a_i$ to $a_{i-1}$. By concatenating these chains, we have a $\delta$-chain from $a_n$ to $a_0$, i.e. from $q$ to $p$. Since we can do this for all $\delta>0$, we see that $p\in\chainacc{q}$.
	
	Thus, we have established that chain equivalence is an equivalence relation on $X$. Now, fix $p\in X$ and suppose that $q$ is in the component of $p$. Let $\delta>0$. As in the proof of Theorem \ref{identity}, we can find find a sequence $\seq{x_j}_{j\leq n}$ with $x_0=p$, $x_n=q$ and $d(x_j,x_{j+1})<\delta/2$ for all $i<n$. Additionally, since $(X,f)$ is recurrent, for each $i<n$, we can find $M_j>0$ with $d(f^{M_j}(x_j),x_j)<\delta>2$.
	
	Now, let $M=\sum_{j<n}M_j$ and define the sequence $\seq{a_i}_{i\leq M}$ as follows. For each $i\leq M$, choose $k\leq n$ such that $\sum_{j< k}M_j\leq i<\sum_{j\leq k}M_j$ and define $a_i=f^{i-\sum_{j< k}M_j}(x_k)$. It is straightforward to verify that $\seq{a_i}_{i\leq M}$ is a $\delta$-chain from $p$ to $q$. Indeed, for each $i<M$, there exists $j$ such that either	$a_i=f^l(x_j)$ for some $l<M_j-1$ or $a_i=f^{M_j-1}(x_j)$. In the former case, $a_{i+1}=f^{l+1}(x_j)$ and hence $d(f(a_i),a_{i+1})=0<\delta$ and in the latter case $a_{i+1}=x_{j+1}$ and $d(f(a_i),a_{i+1})=d(f(f^{M_j-1}(x_j)),x_{j+1})\leq d(f^{M_j}(x_j),x_{j})+d(x_j,x_{j+1})<\delta$, thus verifying the claim that $\seq{a_i}_{i\leq M}$ is a $\delta$-chain.	
\end{proof}

Interestingly, this establishes that for each $p\in X$, the set $\chainacc{p}$ is \emph{chain transitive} in the sense that for any pair $x,y\in\chainacc{p}$, and any $\delta>0$ there is a $\delta$-chain from $x$ to $y$. In particular, this means that $\chainacc{p}$ is closed and completely invariant \cite{deVries}. In addition, this also means that every point belongs to a \emph{basic set} of $(X,f)$, i.e. a maximal chain transitive subset of $X$. As we shall see in the following theorem, if $(X,f)$ also has shadowing, then $\chainacc{p}$ is also \emph{minimal} in the sense that if $A\subseteq\chainacc{p}$ is closed and invariant under $f$ (i.e. $f(A)\subseteq A$), then $A=\chainacc{p}$.

\begin{theorem} \label{chainacc}
	Let $(X,f)$ be a recurrent system with shadowing. Then, for all $p\in X$, $(\chainacc{p},f|_{\chainacc{p}})$ is minimal.
\end{theorem}

\begin{proof}
	Suppose that $(X,f)$ is recurrent and has shadowing. Fix $p\in X$ and let $C\subseteq \chainacc{p}$ be closed and invariant. We will show that $p\in C$. Once we have done so, since Lemma \ref{equiv} establishes that for all $q\in\chainacc{p}$, $\chainacc{q}=\chainacc{p}$, we see that this holds for all $q\in\chainacc{p}$, establishing that $C=\chainacc{p}$.
	
	Now, let $c\in C$ and $\epsilon>0$. Choose $\delta>0$ such that any $\delta$-pseudo-orbit for $f$ is $\epsilon/3$-shadowed. Since $c\in\chainacc{p}$, we can fix a $\delta$-chain $\seq{a_i}_{i\leq n}$ with $a_0=p$ and $a_n=c$. We can extend this $\delta$-chain to a $\delta$-pseudo-orbit for $f$ by defining, for $i>n$, $a_i=f^{i-n}(c)$.
	
	By choice of $\delta$, fix $z\in X$ such that the orbit of $z$ $\epsilon/3$-shadows $\seq{a_i}_{i\in\omega}$. In particular,  $d(z,p)<\epsilon/3$ and for all $i\geq n$, $d(f^i(z),f^{i-n}(c))<\epsilon/3$. Since $(X,f)$ is recurrent, there exists $M\geq n$ such that $d(f^M(z),z)<\epsilon/3$. Then $d(f^{M-n}(c),p)\leq d(f^{M-n}(c),f^M(z))+d(f^{M}(z),z)+d(z,p)<\epsilon$. Since $C$ is invariant by hypothesis, $f^{M-n}(c)\in C$ and thus we see that $d(p,C)<\epsilon$. Bu this holds for all $\epsilon>0$, and $C$ is closed, so $p\in C$ as claimed.
\end{proof}

Interestingly, by applying a theorem of Birkhoff (Theorem 4.2.2 in \cite{deVries}), we can establish that recurrent systems with shadowing actually exhibit a much stronger form of recurrence. In particular, by this theorem, since every point of $X$ belongs to a minimal set, they are also \emph{almost periodic} (a point $x$ is almost periodic provided that for all $\epsilon>0$ there exists $l>0$ such that for any $n\in \omega$, the intersection $\{f^i(x):n\leq i<n+l\}\cap B_\epsilon(x)$ is nonempty).

In Theorems \ref{identity} and \ref{periodic}, we established that shadowing implies total disconnectedness in periodic systems. As it happens, since periodic orbits are minimal, these results are also consequences of the following more general result, which is a direct application of Theorem \ref{chainacc} and Lemma \ref{equiv}. 


\begin{corollary} \label{recurrentanalogue}
	Let $(X,f)$ be a recurrent system with shadowing. Then $X$ is the disjoint union of its minimal subsets and no component of $X$ meets more than one such subset.
\end{corollary}


This has particularly interesting implications when $X$ is connected. 

\begin{corollary} \label{recurrentminimal}
	Let $X$ be connected and let $(X,f)$ be a system with shadowing. Then $(X,f)$ is minimal if and only if it is recurrent.
\end{corollary}

\begin{proof}
	Let $(X,f)$ be a system with shadowing with $X$ connected. 
	
	Suppose that $(X,f)$ is recurrent. Then, since $X$ has only one component, Corollary \ref{recurrentanalogue} immediately implies that $(X,f)$ is minimal.
	
	Conversely, it is well known that a minimal system is recurrent \cite{deVries}, and thus if $(X,f)$ is minimal, then it is recurrent.
\end{proof}


We now turn out attention to the stronger recurrence-type notion of uniform rigidity. In uniformly rigid systems, the `return times' of each point coincide, allowing for much stronger results which paralleling those for periodic systems.


\begin{theorem} \label{rigidsystem}
	A uniformly rigid system $(X,f)$ has shadowing if and only if $X$ is totally disconnected.
\end{theorem}

\begin{proof}
	First, suppose that $(X,f)$ is uniformly rigid and has shadowing. Fix $p\in X$ and let $C$ be the component of $p$ in $X$.
	
	Fix $q\in C$ and let $\epsilon>0$. Since $(X,f)$ has shadowing, fix $\delta>0$ such that every $\delta$-pseudo-orbit is $\epsilon/3$-shadowed. Since $(X,f)$ is uniformly rigid, choose $N$ such that $\rho(f^N,id_X)<\delta/2.$
	
	Now, since $p$ and $q$ belong to the same component of $X$, we can find $K\geq 0$ and a sequence $\seq{a_i}_{i\leq K}$ in $X$ such that $a_0=p$, $a_K=q$ and for all $i<K$, $d(a_i,a_{i+1})<\delta/2$. We can extend this to an infinite sequence $\seq{a_i}_{i\in\omega}$ by defining $a_i=q$ for all $i>K$. From this sequence, we construct the sequence $\seq{c_i}_{i\in\omega}$ by defining, $c_i=f^{i-jN}(a_j)$ where $jN\leq i<(j+1)N$. It is not difficult to see that this is a $\delta$-pseudo-orbit for $f$: for $i$ which are not congruent to $N-1$ modulo $N$, we have $d(f(c_i),c_{i+1})=d(c_{i+1},c_{i+1})=0<\delta$ and otherwise $d(f(c_i),c_{i+1})=d(f(f^{N-1}(a_j)),a_{j+1})=d(f^N(a_j),a_{j+1})\leq d(f^N(a_j),a_{j})+ d(a_j,a_{j+1})<\delta$.
	
	By choice of $\delta$, we can find $z\in X$ which $\epsilon/3$-shadows $\seq{c_i}$. In particular, $d(z,p)<\epsilon/3$, and for all $k\geq K$, $d(q,f^{kN}(z))=d(c_{kN},f^{kN}(z))<\epsilon/3$.
	
	To complete the proof, choose $t\geq K$ such that $\rho(f^t,id_X)<\epsilon/(3N)$. It is easy to verify then, that $\rho(f^{tN},id_X)<\epsilon/3$ and thus 
	\[d(p,q)\leq d(p,z)+d(z,f^{tN}(z))+d(f^{tN}(z),q)<\epsilon.\]
	Since this holds for all $\epsilon>0$, we see that $d(p,q)=0$, i.e. $q=p$ and therefore $C=\{p\}$. Thus, $X$ has no non-degenerate components and is totally disconnected.
	
	Conversely, suppose that $(X,f)$ is uniformly rigid and that $X$ is totally disconnected. Fix $\epsilon>0$. Since $X$ is compact and metric, we can find a finite cover $\mathcal U=\{U_j:j\leq n \}$ of $X$ consisting of pairwise disjoint open sets of diameter less than $\epsilon$.
	
	Fix $0<\eta<\min\{d(U_j,U_k):j\neq k\}$. Now, since $(X,f)$ is uniformly rigid, choose $N>0$ with $\rho(f^N,id_X)<\eta$ and, by uniform continuity, choose $\delta>0$ such that if $d(a,b)<\delta$, then $d(f^i(a),f^i(b))<\eta$ for all $i\leq N$. 
	
	Let $\seq{x_i}$ be a $\delta$-pseudo-orbit. We claim that the orbit of $x_0$ $\epsilon$-shadows $\seq{x_i}$. Indeed, for each $j\in\omega$ and $k\leq N$, we have $d(f^k(f(x_j)),f^k(x_{j+1}))<\eta$ since $d(f(x_i),x_{i+1})<\delta$. It follows that for each $k\leq N$ and $i\in\omega$, the sequence $f^k(x_{i-k}),f^{k-1}(x_{i-k+1}),\ldots x_{i}$ has the property that consecutive terms are within $\eta$ of one another, and hence there exists a unique index $j(i)$ such that each of these terms belongs to $U_{j(i)}$.
	
	We claim that $f^i(x_0)\in U_{j(i)}$ for all $i\in\omega$. Indeed, by taking $k=i$ in the above observation, we see that the claim holds for $i\leq N$. Now, suppose that the claim fails and let $L$ be the least index at which it fails. Since $L>N$, $N-L\geq 0$ and each of $ d\left(x_L,f^N(x_{L-N})\right)$, $d\left(f^N(x_{L-N}),x_{L-N}\right)$, $d\left(x_{L-N},f^{L-N}(x_0)\right)$, and $d\left(f^{L-N}(x_0),f^{L}(x_0)\right)$ is less than $\eta$ (the first by the observation in the previous paragraph, the second and fourth by choice of $N$ and uniform rigidity, and the third by minimality if $L$). Therefore, each of $x_L, f^N(x_{L-N}), x_{L-N}, f^{L-N}(x_0)$ and $f^L(x_0)$ each belong to the same element of $\mathcal U$. In particular, $f^L(x_0)\in U_{j(L)}$, a contradiction.
	
	Thus, for all $i\in\omega$, both $f^i(x)$ and $x_i$ belong to $U_{j(i)}$ and therefore $d(f^i(x_0),x_i)<\epsilon$. This establishes that the orbit of $x_0$ $\epsilon$-shadows $\seq{x_i}$ and thus $(X,f)$ has shadowing.

\end{proof}


\section{Rigid Spaces and Shadowing} \label{Rigid Spaces}

As mentioned in the introduction, there is a significant body of work devoted to determining the prevalence of shadowing in $\cont(X)$ for various categories of topological spaces $X$. In particular, shadowing has been shown to be a generic property of dynamical systems on manifolds (\cite{Yano, Odani, Pilyugin-Plam, Mizera, Mazur-Oprocha}), dendrites (\cite{BMR-Dendrites}), and more exotic locally connected continua (\cite{Meddaugh-Genericity}). In addition, shadowing has been shown to be generic in the space of \emph{surjections} ($\surj(X)$) on spaces in those classes.

In this section, we use the results of Section \ref{Rigid systems} to classify those spaces $X$ which have the property that \emph{every} map (in $\cont(X)$ and in $\surj(X)$) as well as those in which \emph{essentially no} map has shadowing. Of particular note is the demonstration of a compact metric space having the property that shadowing is not a dense property in $\cont(X)$.

We begin by describing a system on the Cantor set which does not have shadowing.

\begin{lemma} \label{cantorexample}
	Let $C$ denote the standard middle-third Cantor set in the interval $[0,1]$ and define $t:C\to C$ by 
	\[t(x)=\begin{cases}
		3x & x\leq1/3\\
		0 & x\geq2/3.
	\end{cases}\]
Then $t$ is a continuous surjection and $(C,t)$ does not have shadowing.
\end{lemma}

\begin{proof}
	It is trivial to see that $t\in\surj(X)$. To see that $(C,t)$ does not have shadowing we proceed as follows.
	
	Fix $1>\epsilon>0$ and choose $\delta>0$. We will now construct a $\delta$-pseudo-orbit for $t$ which cannot be $\epsilon$-shadowed by the orbit of any point. Now, choose $M\in\omega$ such that $1/3^{M}<\delta$ and define the sequence $\seq{p_i}_{i\in\omega}$ by $p_i=\frac{3^{(i\mod{M+1})}}{3^M}$. Observe that if $i\mod{M+1}\neq M$, then $t(p_i)=p_{i+1}$ and if $i\mod{M+1}=M$, then $t(p_i)=0$ and thus $d(p_{i+1},t(p_i))=1/3^M<\delta$, so that $\seq{p_i}$ is a $\delta$-pseudo-orbit for $t$. 
	
	To see that no orbit $\epsilon$-shadows this pseudo-orbit, note that for any $z\in C$, there exists $N\in\omega$ with $t^n(z)=0$ for $n\geq N$. However, if we choose $i\geq N$ with $i\mod{M+1}=M$, then $d(p_i,t^n(z))=d(1,0)\geq\epsilon$.
	
	Thus, there is no $\delta>0$ such that every $\delta$-pseudo-orbit is $\epsilon$-shadowed, and $(C,t)$ does not have shadowing.
\end{proof}

We are now prepared to prove the following.

\begin{theorem} \label{finitesurjcont}
	For a compact metric space $X$, the following are equivalent:
	
	\begin{enumerate}
		\item every map in $\cont(X)$ has shadowing.
		\item every map in $\surj(X)$ has shadowing
		\item $X$ is finite.
	\end{enumerate}  
\end{theorem}

\begin{proof}
	Let $X$ be a compact metric space.
	
	Clearly, if every map in $\cont(X)$ has shadowing, then so does every map in $\surj(X)$.
	
	Now, suppose that each map in $\surj(X)$ has shadowing. In particular, the identity map on $X$ has shadowing, and therefore by Theorem \ref{identity}, $X$ is totally disconnected. We now have two cases to consider---either there are finitely many isolated points in $X$ or there are infinitely many. 
	
	First, suppose that there are finitely many isolated points in $X$ and let $F$ be the set of these points. Then $X\setminus F$ is compact, totally disconnected and has no isolated points, and hence is a Cantor set \cite{Nadler}. Let $(C,t)$ be the system from Lemma \ref{cantorexample} and let $h:X\setminus F\to C$ be a homeomorphism. Define $f:X\to X$ by $f(x)=x$ if $x\in F$ and $f(x)=h^{-1}(t(h(x)))$ if $x\in X\setminus F$. It is easily seen that since $(C,t)$ does not have shadowing, then $(X,f)$ also does not, contradicting our assumption.
	
	Thus, $X$ must have infinitely many isolated points. Since $X$ is compact metric, we can enumerate a sequence $\seq{p_i}_{i\in\omega}$ of distinct isolated points which converges to a point $p_\infty\in X$. We then define $f:X\to X$ by taking $f(x)=x$ if $x\notin\{p_i:i\in\omega\}$, $f(p_0)=p_\infty$, and $f(p_i)=p_{i-1}$ for $i>0$.. It is easy to see that this is a continuous surjection on $X$ and, by an argument similar to that in Lemma \ref{cantorexample}, $(X,f)$ does not have shadowing, again contradicting out assumption, thus establishing that $X$ is finite.	
	
	Finally, assume that $X$ is finite. Let $f\in \cont(X)$ and fix $\epsilon>0$. Let $\delta>0$ such that if $a\neq b\in X$, then $d(a,b)\geq\delta$. Then every $\delta$-pseudo-orbit is a true orbit and is trivially shadowed by its initial point.
\end{proof}

It is worth pointing out that we can prove that $X$ being finite implies that every map $\surj(X)$ has shadowing can also be proven by appealing to the results of the previous section. In particular, if $X$ is finite and $f\in\surj(X)$, then $f$ is a permutation and therefore there exists $N>0$ such that $f^N=id_X$, and therefore by Theorem \ref{periodic}, $f$ has shadowing since $X$ is finite.

%

%

We now turn our attention to the analysis of those spaces in which essentially no map has shadowing. We begin by noting that if $X$ is a compact metric space and $c\in X$, then the constant map $x\mapsto c$ has shadowing, and therefore $\cont(X)$ will always contain at least some maps with shadowing. 

The following examples of Cook will be useful. Through an intricate process using inverse limits to `blow-up' points into solenoids, Cook developed a family of continua with very few non-constant self-maps. We list those properties of these continua which are relevant to our discussion in the following remark.

\begin{remark}[Cook, \cite{Cook}] \label{Cook}
For each $n\in \N$, there exists a non-degenerate continuum $H_n$ such that there exist $n$ and only $n$ elements of $\surj(H_n)$, each of which is a homeomorphism. Furthermore, if $f\in\cont(H_n)\setminus\surj(H_n)$, then $f(H_n)$ is a singleton.

Additionally, there is a continuum $H_\infty$ such that $\surj(H_\infty)$ is homeomorphic to the Cantor set and each element is a homeomorphism. Furthermore, if $f\in\cont(H_\infty)\setminus\surj(H_\infty)$, then $f(H_\infty)$ is a singleton.
\end{remark}

These examples allow us to answer a few open questions concerning shadowing. In particular, in \cite{Meddaugh-Genericity} we asked whether there were, in fact any continua $X$ such that shadowing is not generic in $\cont(X)$.

\begin{theorem} \label{HnNondense}
	There exist continua with the property that shadowing is not dense in $\cont(X)$. In particular, for each $n\in\N$, the continuum $H_n$ has the property that shadowing is not dense in $\cont(X)$.
\end{theorem}

\begin{proof} 
	Fix $n\in \N$ and note that $\surj(H_n)$ has a exactly $n$ elements, each which is isolated in $\cont(H_n)$ since $\cont(H_n)\setminus\surj(H_n)$ contains only the constant maps. In particular, $id_{H_n}$ is one of these maps. Since $H_n$ is not totally disconnected, by Theorem \ref{identity}, $(H_n,id_{H_n})$ does not have shadowing. Since $id_{H_n}$ is isolated in $\cont(H_n)$, we see that shadowing is not dense in $\cont(H_n)$.
\end{proof}

In \cite{meddaugh2021shadowing} we demonstrated that for a large class of spaces, the shadowing property is equivalent to the \emph{continuously generated pseudo-orbit tracing property (CGPOTP)}. Briefly, a system $(X,f)$ has CGPOTP if for all $\epsilon>0$ there exists $\delta>0$ such that if $f,g\in\cont(X)$ with $\rho(f,g)<\delta$, then every $g$-orbit is $\epsilon$-shadowed by an $f$-orbit. It was left open in that paper whether there were any systems with CGPOTP, but not shadowing.

\begin{theorem} \label{HnNoShadowingInSurj}
	For each $n\in\N$, and $f\in\surj(H_n)$, the system $(H_n,f)$ has CGPOTP but does not have shadowing.
\end{theorem}

\begin{proof}
	Fix $n\in \N$ and $f\in\surj(H_n)$.	
	
	To observe that $(H_n,f)$ does not have shadowing, note that $\surj(H_n)$ consists only of homeomorphisms and therefore is a group under composition. Since it is finite, every element is of finite order and, in particular, there exists $k>0$ with $f^k=id_{H_n}$ and thus $(H_n,f)$ is a periodic system . Since $H_n$ is not totally disconnected, $(H_n,f)$ does not have shadowing.
	
	To see that $(H_n,f)$ has CGPOTP, since $f$ is isolated in $\cont(X)$, we can choose $\delta>0$ such that the set $\{g\in\cont(X):\rho(f,g)<\delta\}=\{f\}$. Now, if we fix $\epsilon>0$, then any continuously generated $\delta$-pseudo-orbit for $f$ is a true orbit for $f$, and is therefore trivially shadowed by the orbit of its initial term.
\end{proof}

The observant reader may have noticed that we have made no claims regarding the properties of $H_\infty$. In Theorems \ref{HnNondense} and \ref{HnNoShadowingInSurj}, since $\surj(H_n)$ is finite, we were able to demonstrate that the surjections in $\cont(X)$ were periodic and apply Theorems \ref{identity} and \ref{periodic}. Since $\surj(H_\infty)$ is not finite, this technique cannot work. In order to generalize this, observe that for each $n\in\N$, each map in $\surj(H_n)$ is periodic and hence uniformly rigid. As it happens, this is the relevant property to generalize the preceding results.

\begin{definition}
	A compact metric space $X$ is \emph{rigid} provided that for each map $f\in\surj(X)$, the system $(X,f)$ is uniformly rigid.
\end{definition} 

Examples of rigid spaces include finite spaces as well as the $H_n$ continua of Cook. As we shall see later, the $H_\infty$ continuum is also rigid.

\begin{lemma} \label{totallydisconnectedrigid}
	A totally disconnected space is rigid if and only if it is finite.
\end{lemma}

\begin{proof}
	Let $X$ be a finite space. Since $\surj(X)$ is finite and consists only of homeomorphisms, as in the proof of Theorem \ref{HnNoShadowingInSurj}, each map in $\surj(X)$ is periodic and hence uniformly recurrent. Thus, $X$ is rigid.
	
	Now, assume that $X$ is totally disconnected and infinite. It is easy to verify that a continuous surjection constructed in the fashion of the maps constructed in the proofs of Lemma \ref{cantorexample} and Theorem \ref{finitesurjcont} is not uniformly rigid, as (in both constructions) there exists points $x\neq y\in X$ and $n\in\N$ such that $\{f^n(x): n\geq N\}=\{y\}$, so that no subsequence of $\seq{f^n}$ can converge to $id_X$.
\end{proof}
%
%

By Theorem \ref{finitesurjcont}, then, finite rigid spaces have the property that \emph{every} system has shadowing. This stands in stark contrast to the infinite case in which \emph{no} surjective system has shadowing.

\begin{theorem} \label{rigidshadowing}
	Let $X$ be an infinite rigid space and $f\in\surj(X)$. Then $(X,f)$ does not have shadowing.
\end{theorem}

\begin{proof}
	Let $X$ be an infinite rigid space and $f\in\surj(X)$. By Lemma \ref{totallydisconnectedrigid}, $X$ is not totally disconnected. Since $X$ is rigid, $(X,f)$ is uniformly rigid, and since $X$ is not totally disconnected, Theorem \ref{rigidsystem} gives us that $(X,f)$ does not have shadowing.
\end{proof}

Of course, it is not immediately apparent how to determine whether a space $X$ is rigid. Fortunately, compactness of $\surj(X)$ is sufficient. 

\begin{theorem} \label{rigidcompact}
	Let $X$ be a compact metric space. If $\surj(X)$ is compact, then $X$ is rigid.
\end{theorem}

\begin{proof}
	Suppose $X$ is a compact metric space with $\surj(X)$ compact and fix $f\in\surj(X)$. 
	
	For each $g\in\surj(X)$, define $I(g)=\overline{\{g^n:n\in\N\}}$ and note that this is a compact subset of $\surj(X)$. We claim that $I(g)$ is closed under composition in the sense that if $p,q\in I(g)$, then $p\circ q\in I(g)$. To observe this, let $p,q\in I(g)$ and fix $\epsilon>0$. Choose $\seq{n_i}_{i\in\omega}$ and $\seq{m_i}_{i\in\omega}$ so that $\seq{g^{n_i}}$ converges to $p$ and $\seq{g^{m_i}}$ converges to $q$ (note that one or both of these may be constant sequences). 
	Now, choose $\delta>0$ (by continuity of $p$) so that if $d(a,b)<\delta$, $d(p(a),p(b))<\epsilon/2$. Finally, choose	 $K\in\omega$ such that for all $i\geq K$, we have both $\rho(g^{n_i},p)<\epsilon/2$ and $\rho(g^{m_i},q)<\delta$. Then, for all $i\geq K$ and $x\in X$, we have
	\[d(g^{n_i+m_i}(x),p(q(x)))\leq d(g^{n_i}(g^{m_i}(x)), p(g^{m_i}(x)))+d(p(g^{m_i}(x)), p(q(x)))<\epsilon,\]
	and therefore $\rho(g^{n_i+m_i},p\circ q)\leq\epsilon$. It follows that $p\circ q\in \overline{\{g^n:n\in\N\}} = I(g)$.
	
	Since $I(g)$ is closed under composition in this sense, if $h\in I(g)$, then $\{h^n:n\in\N\}\subseteq I(g)$, and therefore $I(h)\subseteq I(g)$ since $I(g)$ is closed.

	We will now show $id_X\in I(f)$. As a consequence of the above, the set $\mathcal I(f)=\{I(h):h\in I(f)\}$ is partially ordered by inclusion. Note that, by compactness of $\surj(X)$ and hence of $I(f)$, if $\seq{h_i}_{i\in\omega}$ is a sequence in $I(f)$ such that $I(h_{i+1})\subseteq I(h_i)$ for all $i\in\omega$, then $\bigcap I(h_i)$ is nonempty. In particular, there exists $g\in \bigcap I(h_i)$ and $I(g)$ is a subset of $I(h_i)$ for each $i\in\omega$. By Zorn's lemma, since every descending chain has a lower bound, there exists a minimal element of $\mathcal I(f)$.
	
	Fix $g\in I(f)$ such that $I(g)$ is minimal in $\mathcal I(f)$. Observe that $I(g^2)\subseteq I(g)$ and $I(g^2)\in\mathcal I(f)$, so $I(g^2)=I(g)$ and, in particular, $g\in I(g^2)$, and therefore $g\in\overline{\{g^{1+n}:n\in\N\}}$. It follows that there exists  a sequence $\seq{n_i}_{i\in\omega}$ in $\N$ with $g^{1+n_i}$ converging to $g$ (again, this may be a constant sequence). 
	
	Now, since $\surj(X)$ is compact, there exists an increasing sequence $\seq{i_j}_{j\in\omega}$ such that $\seq{g^{n_{i_j}}}_{j\in\omega}$ converges to a function $h$. But this function must satisfy $h\circ g=g$, since the sequence $\seq{g^{n_{i_+j}}\circ g}=\seq{g^{1+n_{i_j}}}$ converges to $g$. Since $g$ is a surjection, it follows that $h=id_X$. This establishes that $id_X\in I(f)$ and therefore $f$ is uniformly rigid.
\end{proof}

Since $\surj(H_\infty)$ is a Cantor set, it is compact and thus $H_\infty$ is rigid (and infinite). Notice that $\cont(H_\infty)\setminus\surj(\infty)$ contains only the constant maps. Since the collection of constant maps is homeomorphic to $H_\infty$, it is compact, and hence closed. Thus, we have that $\surj(H_\infty)$ is open, which yields the following corollary.

\begin{corollary}
	Shadowing is not dense in $\cont(H_\infty)$ and is not present in $\surj(H_\infty)$.
\end{corollary}

\section{Rigidity and Shadowing for Non-Surjections} \label{non-surjective}

As an interesting consequence of Theorem \ref{rigidcompact}, if $X$ is a space with $\cont(X)$ compact, then $X$ is rigid (since $\surj(X)$ is closed in $\cont(X)$, and therefore compact) and thus shadowing is not present in $\surj(X)$. A priori, however, $\surj(X)$ may be nowhere dense in $\surj(X)$ and therefore the question of whether shadowing is dense in $\cont(X)$ is still open. In order to better understand what can happen when $\cont(X)$ is compact, we first generalize the notion of uniform rigidity.

\begin{definition}
	A dynamical system $(X,f)$ is \emph{uniformly retract-rigid} provided that there exists a retraction $r:X\to R$ and a subsequence of $\seq{f^i}$ which converges uniformly to $r$.
\end{definition}

It is not difficult to show that a map $r$ is a retraction if and only if $r$ is \emph{idempotent}, i.e., $r\circ r=r$, and thus uniformly retract-rigid systems can be alternatively characterized as those systems which have a subsequence of iterates converging to an idempotent system. It is also worth pointing out that it must be the case that the set $R$ is $f$-invariant, in the sense that $f(R)=R$ and that it must be the case that $\bigcap f^n(X)= R$. Indeed, fix $y\in R$ and $n\in\omega$. Notice that the sequence $\seq{f^{n_i}(y)}$ converges to $y$, and since for each $n$, $f^n(X)$ is a compact subset of $X$ which contains a tail of $\seq{f^{n_i}(y)}$, we have $y\in f^n(X)$, establishing that $R\subseteq\bigcap f^n(X)$. Conversely, fix $\eta>0$ and choose $N\in\omega$ such that for $i\geq N$, and $x\in X$ we have $f^{n_i}(x)\in B_\epsilon(r(x))$, and by compactness, $f^{n_i}(X)\subseteq B_\epsilon(R)$. In particular, $\bigcap f^n(X)\subseteq f^{n_i}(X)\subseteq B_\epsilon(R)$, and since this holds for every $\epsilon>0$, we have $\bigcap f^n(X)\subseteq R$ by compactness.

The remainder of this paper focuses on developing results for uniform retract-rigidity which parallel those we have developed for uniform rigidity.

\begin{theorem} \label{retractrigid}
	If the uniformly retract-rigid system $(X,f)$ has shadowing, then $\bigcap f^n(X)$ is totally disconnected.
\end{theorem}

\begin{proof}
	Since $(X,f)$ is uniformly retract-rigid, we can find a compact set $R\subseteq X$ and a retraction $r:X\to R$ and and a subsequence $\seq{n_i}$ such that $\seq{f^{n_i}}$ converges to $r$. 
	
	Now, suppose that $(X,f)$ has shadowing. Fix $\epsilon>0$ and choose $\delta>0$ such that every $\delta$-pseudo-orbit is $\epsilon/4$-shadowed. Now, let $p\in R$ and let $C$ be the component of $p$ in $R$.
	
	Fix $q\in C$. We will show that $q=p$ in  a manner similar to that of Theorem \ref{rigidsystem}. Towards this end, choose $N$ such that $\rho(f^N,r)<\delta/2$.
	
	Now, since $p$ and $q$ belong to the same component of $X$, we can find $K\geq 0$ and a sequence $\seq{a_i}_{i\leq K}$ in $X$ such that $a_0=p$, $a_K=q$ and for all $i<K$, $d(a_i,a_{i+1})<\delta/2$ and we extend this to an infinite sequence $\seq{a_i}_{i\in\omega}$ by defining $a_i=q$ for all $i>K$. From this sequence, we construct the sequence $\seq{c_i}_{i\in\omega}$ by defining, $c_i=f^{i-jN}(a_j)$ where $jN\leq i<(j+1)N$. As in Theorem \ref{rigidsystem}, this is a $\delta$-pseudo-orbit for $f$.
	
	Moreover, since $f(R)=R$, we can, for each $k\in\omega$, find $c_{-n_k}\in R$ with $f^{n_k}(c_{-n_k})=c_0$. By defining $c_{j-n_k}=f^j(c_{-n_k})$ for $j\leq n_k$, we can construct a family of $\delta$-pseudo-orbits $\seq{c_i}_{i\geq-n_k}$ in $R$.
	
	By choice of $\delta$, the pseudo-orbit $\seq{c_i}_{i\geq-n_k}$ is $\epsilon/4$-shadowed by some point $z_k\in X$, and therefore the pseudo-orbit $\seq{c_i}_{i\in\omega}$ is $\epsilon/4$-shadowed by $f^{n_k}(z_k)$. By passing to a subsequence if necessary, we can assume that the points $f^{n_k}(z_k)$ converge to some point $z\in X$. This point is easily seen to $\epsilon/3$-shadow  $\seq{c_i}_{i\in\omega}$, since $d(f^i(f^{n_k}(z_k)),c_i)$ converges to $d(f^i(z),c_i)$. Furthermore, this point $z$ belongs to $R$ since $R$ is closed and  $d(z,R)$ is the limit of $d(f^{n_k}(z_k),R)$, which is clearly zero.
	
	To complete the proof, choose $t\geq K$ such that $\rho(f^t,r)<\epsilon/(3N)$. It is easy to verify then, that $\rho(f^{tN},r)<\epsilon/3$ and thus 
	\[d(p,q)\leq d(p,z)+d(z,f^{tN}(z))+d(f^{tN}(z),q)<\epsilon.\]
	As in Theorem \ref{rigidsystem}, it quickly follows that $C=\{p\}$, i.e., $R$ is totally disconnected.
\end{proof}

Note that, in contrast with the result of Theorem \ref{rigidcompact}, this is not a complete characterization. However, it possible that more can be said. In particular, if $R=\bigcap f^n(X)$ is totally disconnected, then $(R,f|R)$ has shadowing since the restriction of $f$ to $R$ is a rigid system. However, the action of $f$ on the complement of $R$ may not be sufficiently well structured to guarantee shadowing of $(X,f)$. Indeed, there are examples of uniformly retract-rigid systems in which $(X,f)$ does not have shadowing. This can be observed by letting $(Y,g)$ be a system without shadowing and letting $(X,f)$ be the `cone' of $(Y,g)$ with the system $([0,1],t\mapsto \frac{1}{2}t)$, i.e. the space $Y$ is the quotient of $Y\times[0,1]$ obtained by identifying all points of the form $(y,0)$ and $f((y,t))$ is defined to be $(g(y),\frac{1}{2}t)$. Still, this system exhibits the property of \emph{eventual shadowing} and it seems probable that every uniformly retract-rigid system with $R$ totally disconnected exhibits this property---see \cite{GoodMeddaugh-ICT} for a detailed discussion of this concept. 

As is the case with uniform rigidity, it is useful to consider the class of spaces on which every dynamical system is uniformly retract-rigid.

\begin{definition}
	A compact metric space $X$ is \emph{retract-rigid} provided that for every $f\in\cont(X)$, the system $(X,f)$ is uniformly retract-rigid.
\end{definition}

Note that a retract-rigid space is necessarily rigid, but a rigid space need not be retract-rigid.

We will now show that in systems on retract-rigid spaces, shadowing is an uncommon property in $\cont(X)$. We remark that since each constant map in $\cont(X)$ has shadowing, shadowing can never be absent in $\cont(X)$ like it can be in $\surj(X)$. In order to proceed, we first need the following lemmas.

\begin{lemma} \label{retractions}
	If $X$ is retract-rigid and $r:X\to R$ is a retraction, then $R$ is rigid.
\end{lemma}

\begin{proof}
	Suppose that $X$ is retract-rigid and let $r:X\to R$ be a retraction. 
	
	Suppose that $f\in\surj(R)$. Then $g=f\circ r\in\cont(X)$ and therefore $(X,g)$ is retract-rigid. Fix $\seq{n_i}$ such that $\seq{g^{n_i}}$ converges to a retraction $r'$. Since $g^{n_i}(X)=f^{n_i}(r(X))=f^{n_i}(R)=R$ for all $i\in\omega$, it follows that $r'(X)=R$.
	
	But then, $\seq{f^{n_i}}=\seq{g^{n_i}|_{R}}$ converges to $r'|_R=id_R$, and therefore $(R,f)$ is uniformly rigid. Since this holds for all $f\in\surj(R)$, $R$ is a rigid space.
\end{proof}

\begin{lemma} \label{finitelymanycomponents}
	If $X$ is retract-rigid, then $X$ has finitely many components.
\end{lemma}

\begin{proof}
	Suppose that $X$ is retract-rigid and that $X$ has infinitely many components. We will show that there is a retraction of $X$ onto an infinite, totally disconnected space. By Lemma \ref{totallydisconnectedrigid}, such a space cannot be rigid, which by Lemma \ref{retractions} contradicts the assumption that $X$ is retract-rigid, completing the proof.
	
	The claimed retraction is constructed as follows. Since $X$ has infinitely many components, we can find a separation $A_1, B_1$ of $X$, at least one of which contains infinitely many components--without loss of generality, $B_1$. We now proceed inductively: having defined $A_n$ and $B_n$ to be a separation of $B_{n-1}$ with $B_n$ containing infinitely many components, we choose a separation $A_{n+1}$, $B_{n+1}$ of $B_n$ such that $B_{n+1}$ contains infinitely many components. Finally, we define $A_\infty=\bigcap B_{n}$.
	
	Note that the collection $\{A_n:n\in\omega\cup\{\infty\}\}$ is a a pairwise disjoint collection of nonempty closed subsets of $X$, and that each element other than $A_\infty$ is also open. Now, choose a sequence $\seq{a_n}_{n\in\omega}$ with $a_n\in A_n$ for each $n\in\omega$. Since $X$ is compact, we may choose a subsequence $\seq{n_i}_{i\in\omega}$ which converges to a point $a_\infty$, which is necessarily in $A_\infty$. It is easy to see that $\{a_{n_i}:i\in\omega\}\cup\{a_\infty\}$ is infinite, compact and totally disconnected.
	
	Finally, we define $r:X\to\{a_{n_i}:i\in\omega\}\cup\{a_\infty\}$ as follows. For $x\in A_\infty$, $r(x)=a_\infty$. For $x\in A_k$, choose $i\in\omega$ minimal such that $x\leq n_i$ and define $r(x)=a_{n_i}$. All that remains to verify is that $r$ is continuous. To see this, we observe that the subspace topology on $\{a_{n_i}:i\in\omega\}\cup\{a_\infty\}$ has basis consisting of all sets of the form $\{a_{n_i}:i\in\omega\}$ or $\{a_{n_i}:i\geq L\}\cup\{a_\infty\}$. Clearly $r^{-1}(\{a_{n_i}\})=\bigcup_{n_{i-1}<k\leq n_i} A_k$ and $r^{-1}(\{a_{n_i}:i\geq L\}\cup\{a_\infty\})=B_{n_L-1}$ are both open, and thus $r$ is continuous.

\end{proof}


\begin{theorem} \label{infiniteequalsnoshadowing}
	Let $X$ be a retract-rigid space and $f\in\cont(X)$ such that $\bigcap f^n(X)$ is infinite. Then $(X,f)$ does not have shadowing.
\end{theorem}

\begin{proof}
	Suppose $X$ is retract-rigid and that $(X,f)$ has shadowing, but that $\bigcap f^n(X)$ is infinite. Since $X$ is retract-rigid, $(X,f)$ is uniformly retract-rigid and therefore $\bigcap f^n(X)$ is totally disconnected by Theorem \ref{retractrigid}. 
	
	However, by Lemma \ref{finitelymanycomponents}, $X$ has finitely many components. Let $N$ be the number of components. Then, for each $n\in\omega$, $f^n(X)$ has no more than $N$ components since it is the continuous image of $X$. It follows then, that $\bigcap f^n(X)$ has at most $N$ many components.
	
	Thus $\bigcap f^n(X)$ is totally disconnected with finitely many components, and is therefore finite, contradicting our assumption.
\end{proof}

Of course, these results are only useful if we can determine that a space is retract-rigid. Fortunately, the results concerning compactness of $\surj(X)$ and rigidity of $X$ in Theorem \ref{rigidcompact} generalize in the obvious way.

\begin{theorem}
	Let $X$ be a compact metric space. If $\cont(X)$ is compact, then $X$ is retract-rigid.
\end{theorem}

\begin{proof}
	Suppose that $X$ is a compact metric space with $\cont(X)$ compact and fix $f\in\cont(X)$.
	The proof of this result proceeds exactly in the proof of Theorem \ref{rigidcompact} until we reach that $h\circ g = g$. Since $g$ need not be a surjection, we do not know that $h=id_X$. However, since $h\circ g=g$, $h(X)=g(X)$ and $h|_{g(X)}=id_{g(X)}$, so that $h$ is a retract. Since $h\in I(f)$, there is a subsequence of $\seq{f^{n}}$  which converges to a retraction, and thus $f$ is uniformly retract-rigid.
\end{proof}

We close with the following result, which characterizes shadowing in compact connected spaces $X$ with compact $\cont(X)$ (of which the spaces $H_n$ for $n\in\N\cup\{\infty\}$ are trivial examples).

\begin{corollary} \label{compactC}
	Let $X$ be a compact, connected metric space with $\cont(X)$ compact. Then $(X,f)$ has shadowing if and only if $\seq{f^n}$ converges to a constant map.
\end{corollary}

\begin{proof}
	Suppose $X$ is a compact, connected metric space with $\cont(X)$ compact. By the previous result, $X$ is retract-rigid, and thus $(X,f)$ is uniformly retract-rigid.
	
	Then $(X,f)$ has shadowing only if $\bigcap f^n(X)$ is finite, by Theorem \ref{infiniteequalsnoshadowing}. Since $X$ is connected, this is equivalent to $\bigcap f^n(X)$ being a singleton set $\{c\}$ for some $c\in X$. By compactness of $X$, it follows that this is the case if and only if $\seq{f^n}$ converges to the constant map $x\mapsto c$.
	
	Conversely, if $\seq{f^n}$ converges to the constant map $c:X\to X$ given by $x\mapsto c$, we can observe that $(X,f)$ has shadowing as follows. Fix $\epsilon>0$ and choose $N$ such that for $n\geq N$, we have $\rho(f^n,c)<\epsilon/3$. Then, by uniform continuity of $f$, choose $\delta>0$ such that if $d(a,b)<\delta$, then $d(f^i(a),f^i(b))<\frac{\epsilon}{3N}$ for each $i\leq N$.
	
	Now, let $\seq{x_i}$ be a $\delta$-pseudo-orbit. We claim that it is shadowed by the orbit of $x_0$. Indeed, by choice of $\delta$, we have that for all $i\geq0$ and $M\leq N$,
	\begin{align*}d(f^M(x_i), x_{M+i})&\leq\sum_{j=0}^{M-1}d\left(f^{M-j}(x_{i+j}),f^{M-j-1}(x_{i+j+1})\right)\\
	&\leq\sum_{j=0}^{M-1}d\left(f^{M-j-1}(f(x_{i+j})),f^{M-j-1}(x_{i+j+1})\right)\\
	&< M\frac{\epsilon}{2N}=\epsilon/3.
\end{align*}

And thus, by taking $i=0$ and $N=k$, we see that $d(f^k(x_),x_k)<\epsilon/3$ for $k\leq N$. For $k>N$, we let $i=k-N$ and observe
\begin{align*}d(f^k(x_0),x_k)&=d(f^{N+i}(x_0), x_{N+i})\\
	&\leq d(x_{N+i},f^N(x_i))+d(f^N(x_i),c)+d(f^{N+i}(x_0), c)<\epsilon,\end{align*}
thus verifying that the orbit of $x_0$ indeed shadows $\seq{x_i}$.

\end{proof}

\bibliographystyle{plain}
\bibliography{../../ComprehensiveBib}

\end{document}